\documentclass[twoside,11pt,reqno]{amsart}

\usepackage{amsmath,amsthm,amssymb,amstext,amsfonts,amscd,mathrsfs}
\usepackage{graphicx}
\usepackage{xcolor}
\usepackage{multirow}
\usepackage{placeins}
\newtheorem{theorem}{Theorem}
\theoremstyle{plain}

\newtheorem{case}{Case}
\newtheorem{claim}{Claim}

\newtheorem{definition}{Definition}
\newtheorem{example}{Example}

\newtheorem{note}{Note}

\newtheorem{remark}{Remark}

\numberwithin{equation}{section}

\begin{document}
\title[\hfil Note on Extensions of The Beta function]
{Note on Extensions of The Beta function}

\author{Mehar Chand}
\address{Department of Applied Sciences, Guru Kashi University, Bathinda-1513002 (India)}
\email{mehar.jallandhra@gmail.com}




\begin{abstract}The classical beta function $B(x,y)$ is one of the most fundamental special functions, due to its important role in various fields in the mathematical, physical, engineering and statistical sciences. Useful extensions of the classical Beta function  has been considered by many authors. In the present paper, our main objective is to study the convergence of extensions of classical beta function and introduce modified extension of classical beta function. It is interpreted numerically and geometrically in the view of convergence, further properties and integral presentations are established.\end{abstract}

\maketitle

\section{Introduction and Preliminaries}

In many areas of applied mathematics, various types of special functions become essential tools for scientists and engineers. The (Euler's) classical beta function $B(x,y)$ is one of the most fundamental special functions, because of its important role in various fields in the mathematical, physical, engineering and statistical sciences. During last four decades or so, several interesting and useful extensions of  the familiar special functions (such as the Gamma and Beta functions, the Gauss hypergeometric function, and so on) have been considered
by many authors. In the present work, we are concerned with various generalizations of the classical beta function, which can be found in the literature (see \cite{Srivastava-etinkaya-K?ymaz-2014,Parmar-2013,Ozergin-Ozerslan-Altin-2011,Chaudhry-Qadir-Rafique-Zubair-1997}).
\newline

 Recently the generalized beta function $B_p^{\delta, \zeta; \kappa, \mu}(\alpha,\beta)$ is introduced by Srivastava et al., which is the most generalized extension of classical beta function and  is defined as(see\cite{Srivastava-etinkaya-K?ymaz-2014}):

\begin{eqnarray}\label{B1} \aligned &
B_{p}^{(\delta,\zeta; \kappa, \mu)}(\alpha,\beta)\\&
=\int_{0}^{1}t^{\alpha-1}(1-t)^{\beta-1} {}_{1}F_{1}\left(\delta; \zeta; -\frac{p}{t^{\kappa}(1-t)^{\mu}} \right)dt,\endaligned
\end{eqnarray}

where $\Re(p)\geq 0; \min[\Re(\alpha), \Re(\beta), \Re(\delta), \Re(\zeta)]>0; \min[\Re(\kappa), \Re(\mu)]>0$ and ${}_1F_1(.)$ is confluent hypergeometric function, which is special case of the well known generalized hypergeometric series $_pF_q(.)$.
\newline

 The generalized hypergeometric series ${}_p F_q(p, q \in \mathbb{N})$ is defined as (see \cite[p.73]{Rainville-1971}) and \cite[pp. 71-75]{Srivastava-Choi-2012}:

\begin{eqnarray}\label{pFq} \aligned &  {}_pF_q \left[
\begin{array}{cc}\alpha_1,\ldots,\alpha_p; \\ \beta_1,\ldots,\beta_q;\end{array}z\right]=\sum_{n=0}^\infty \, \frac{(\alpha_1)_n \cdots (\alpha_p)_n}{ (\beta_1)_n \cdots (\beta_q)_n} \frac{z^n}{n!}\\&\hskip 4mm={}_pF_q (\alpha_1,\,\ldots,\,\alpha_p ;\, \beta_1,\,\ldots,\,\beta_q;\,z),\endaligned \end{eqnarray}

where $(\lambda)_n$ is the Pochhammer symbol defined (for $\lambda \in \mathbb C$) by (see\cite[p.2 and p.5]{Srivastava-Choi-2012}):

\begin{eqnarray}\label{PS} \aligned & (\lambda)_n :=\left\{\begin{array}{cc} 1 \hskip 38 mm (n=0) \\ \lambda (\lambda +1) \ldots (\lambda+n-1) \quad (n \in {\mathbb N})
\end{array}\right. \endaligned \end{eqnarray}
\begin{eqnarray} \label{PS1}\aligned &
= \frac{\Gamma (\lambda +n)}{ \Gamma (\lambda)}
      \quad (\lambda \in {\mathbb C} \setminus {\mathbb Z}_0^-),\endaligned \end{eqnarray}

and ${\mathbb Z}_0^-$ denotes the set of Non-positive integers and $\Gamma(\lambda)$ is familiar Gamma function.
\newline

When $\kappa=\mu$, \eqref{B1} reduces to the generalized extended beta function $B_{p}^{(\delta, \zeta; \mu)}(\alpha,\beta)$ defined by (see \cite[p. 37]{Parmar-2013})
\begin{eqnarray}\label{B2} \aligned &
B_{p}^{(\delta,\zeta; \mu)}(\alpha,\beta)\\&
=\int_{0}^{1}t^{\alpha-1}(1-t)^{\beta-1} {}_{1}F_{1}\left(\delta; \zeta; -\frac{p}{t^{\mu}(1-t)^{\mu}} \right)dt,\endaligned
\end{eqnarray}
where $\Re(p)\geq 0; \min[\Re(\alpha), \Re(\beta), \Re(\delta), \Re(\zeta)]>0; \Re(\mu)>0$.
\newline

The special case of \eqref{B2}, when $\mu=1$ reduces to the generalized beta type function as follows (see\cite[p.4602]{Ozergin-Ozerslan-Altin-2011})

\begin{eqnarray}\label{B3} \aligned &
B_{p}^{(\delta,\zeta)}(\alpha,\beta)=B_{p}^{(\delta,\zeta; 1)}(\alpha,\beta)\\&
=\int_{0}^{1}t^{\alpha-1}(1-t)^{\beta-1} {}_{1}F_{1}\left(\delta; \zeta; -\frac{p}{t(1-t)} \right)dt,\endaligned
\end{eqnarray}
where $\Re(p)\geq 0; \min[\Re(\alpha), \Re(\beta), \Re(\delta), \Re(\zeta)]>0$.
\newline

The further special case of \eqref{B3} when $\delta=\zeta$ is given due to Choudhary et. al.\cite{Chaudhry-Qadir-Rafique-Zubair-1997} by

\begin{eqnarray}\label{ex-beta} \aligned &
B_{p}(\alpha,\beta)=B(\alpha,\beta;p)=B_{p}^{(\delta,\delta)}(\alpha,\beta)\\&\hskip 2mm
=\int_{0}^{1}t^{\alpha-1}(1-t)^{\beta-1} \exp\left(-\frac{p}{t(1-t)} \right)dt, \\&\hskip 35mm (\Re(p)\geq 0).\endaligned
\end{eqnarray}

When we choose $p=0$, all the above extensions reduces to the classical beta function $B(x, y)$, which is defined by

\begin{eqnarray}\label{B5} \aligned &
B(\alpha,\beta)=\int_{0}^{1}t^{\alpha-1}(1-t)^{\beta-1}dt,\\& \hskip 30mm (\Re(\alpha)>0, \Re(\beta)>0).\endaligned \end{eqnarray}

It is clear to see the following relationship between the classical beta function
$B(\alpha,\beta)$ and its extensions:

\begin{eqnarray}\aligned & B(\alpha,\beta) = B_0(\alpha,\beta) = B_0^{(\delta,\zeta)}(\alpha,\beta)\\&\hskip 20mm= B_0^{(\delta,\zeta;1)}(\alpha,\beta)= B_0^{(\delta,\zeta;1,1)}(\alpha,\beta).\endaligned\end{eqnarray}

 In particular, Euler's beta function $B(x, y)$ has a close relationship to his gamma function,

\begin{eqnarray} \aligned & B(\alpha,\beta)=B(\beta,\alpha)=\frac{\Gamma(\alpha)\Gamma(\beta)}{\Gamma(\alpha+\beta)}. \endaligned \end{eqnarray}

\section{Extensions of Classical Beta Function do not convergent}

We claim that the above extensions of classical beta function in equations \eqref{B1}, \eqref{B2}, \eqref{B3}, and  \eqref{ex-beta} are  not convergent. We prove it mathematically as well as numerically tested in Section \ref{section-3.1}.
\newline

  First we prove that the extension in equation \eqref{ex-beta} is not convergent as follows.
\begin{claim} If $\Re(\alpha)>0,\Re(\beta)>0, \Re(p)>0$, then the extensions of classical beta function in equation \eqref{ex-beta} does not convergent.  \end{claim}

\begin{proof} To prove our claim, we  write the exponential function in series form, the equation \eqref{ex-beta} reduces to the following form

\begin{eqnarray}\label{p-claim-1} \aligned & B(\alpha,\beta;p)=\sum_{n=0}^\infty \frac{(-p)^n}{n!}\int_0^\infty t^{\alpha-n-1}(1-t)^{\beta-n-1}dt,  \endaligned \end{eqnarray}

further using the definition of the beta function in the above equation \eqref{p-claim-1}, we have

\begin{eqnarray}\label{p-claim-1-a} \aligned & B(\alpha,\beta;p)=\sum_{n=0}^\infty \frac{(-p)^n}{n!}B(\alpha-n,\beta-n)\\&
=B(\alpha,\beta)+\frac{-p}{1!}B(\alpha-1,\beta-1)+...\\& +\frac{(-p)^n}{n!}B(\alpha-n,\beta-n)+... \rm{to~infinity~terms}.  \endaligned \end{eqnarray}

 In the above equation \eqref{p-claim-1-a}, $B(\alpha,\beta;p)$ is an power series involving $B(\alpha-n,\beta-n)$ (where $n=0,1,2,...$). The series $B(\alpha,\beta;p)$ will be convergent if $B(\alpha-n,\beta-n)$ is convergent and $B(\alpha-n,\beta-n)$ will be convergent only if $\Re(\alpha-n)>0$ and $\Re(\beta-n)>0$. But this is not possible as $\alpha$  and $\beta$ are finite and $n\rightarrow\infty$. The series $B(\alpha,\beta;p)$ contains many terms, which does not exist. Moreover $B(\alpha,\beta;p)$ is the sum of terms involving $B(\alpha-n,\beta-n)$, which are not convergent, which implies that  $B(\alpha,\beta;p)$ is not convergent.
\end{proof}

\begin{example} If we choose $\alpha=5,\beta=7,p=3$, then from equation \eqref{p-claim-1-a}, we have
\begin{eqnarray}\label{example-1} \aligned & B(5,7;3)=\sum_{n=0}^\infty \frac{(-3)^n}{n!}B(5-n,7-n), \endaligned \end{eqnarray}
 The above series $B(5,7;3)$ is the sum of terms involving  the beta function $B(5-n,7-n)$ and $B(5-n,7-n)$ does not exit for $n>5$. Therefor $B(5,7;3)$ is not convergent.
\end{example}

In the following section, we introduce modified extension of classical beta function in equation \eqref{ex-beta-m}. Further, the extension of classical beta function \eqref{ex-beta}, Modified extension of classical beta function \eqref{ex-beta-m} and the classical beta function are compared and tested numerically to test the convergence of extensions of the classical beta function.


\section{Modified extension of beta function}
In this section, we introduce modified extension of classical beta function (MECBF). Its convergence is proved mathematically, then numerical results are established and compared the results with that of the classical beta function and extension of classical beta function.
\newline

We introduce modified extension of classical beta function as follows
\begin{eqnarray}\label{ex-beta-m} \aligned & B_m(\alpha,\beta)=B(\alpha,\beta;m)\\&\hskip 10mm=\int_0^1 t^{\alpha-1}(1-t)^{\beta-1}e^{mt(1-t)}dt, \endaligned \end{eqnarray}
where $\Re(\alpha)>0,\Re(\beta)>0, m\in\mathbb{C};|m|<M$ (where $M$ is positive finite real number).
\newline

In our investigation to test the convergence of the above extension of the classical beta function  the following definitions are required.

\begin{definition}[Ratio Test\cite{green-1958}]\label{ratiotest} Let $\sum_{n=1}^\infty p_n$ be a series of positive terms, and suppose that $\displaystyle\frac{p_{n+1}}{p_n}\rightarrow {\rm a\,\,limit\,\,} l$ as $n\rightarrow \infty$. Then (i) If $l<1$ the series $\sum p_n$ is convergent, (ii) If $l>1$, then the series is divergent. If $l=1$, the ratio test fails, and the question of convergence of $\sum p_n$ must be investigated by some other methods. \end{definition}

\begin{definition}[Leibniz's Test or Alternating Series Test\cite{green-1958}]\label{Leibniz} The series $\sum_{n=1}^\infty(-1)^{n+1}a_n$ (where $a_n>0$) is convergent provided (i) $a_n$ is a decreasing sequence, (ii) $a_n\rightarrow 0$ as $n\rightarrow\infty$. \end{definition}

\begin{theorem} If $\Re(\alpha)>0,\Re(\beta)>0, m\in\mathbb{C};|m|<M$ (where $M$ is positive finite real number), then the modified extension of the classical beta function in equation \eqref{ex-beta-m} is convergent.  \end{theorem}

\begin{proof} We can write the above equation \eqref{ex-beta-m} as follows

\begin{eqnarray}\label{p-claim-2} \aligned & B(\alpha,\beta;m)=\int_0^1 t^{\alpha-1}(1-t)^{\beta-1}\sum_{n=0}^\infty \frac{[mt(1-t)]^n}{n!}dt\\&
\hskip 10mm=\sum_{n=0}^\infty \frac{(m)^n}{n!}\int_0^1 t^{\alpha+n-1}(1-t)^{\beta+n-1}dt,  \endaligned \end{eqnarray}

further, using the definition of beta function, the  above equation \eqref{p-claim-2} reduces to

\begin{eqnarray}\label{p-claim-2-a} \aligned & B(\alpha,\beta;m)=\sum_{n=0}^\infty \frac{(m)^n}{n!}B(\alpha+n,\beta+n).  \endaligned \end{eqnarray}
 In the above equation, $B(\alpha,\beta;m)$ is in series form involving $B(\alpha+n,\beta+n)$ (where $n=0,1,2,...$) and in each term  of the series, $B(\alpha+n,\beta+n)$ is convergent since $\Re(\alpha+n)>0$ and $\Re(\beta+n)>0$ for $\Re(\alpha),\Re(\beta)>0$, which implies that each term of the series \eqref{p-claim-2-a} exist.
\newline

Now we shall prove  that $B(\alpha,\beta;m)$ is convergent. $m$ may be greater than or less than $0$, so there are two cases as follows.

\begin{case}\label{case-1} If $m>0$, then to prove $B(\alpha,\beta;m)$ is convergent. \end{case}
 The equation \eqref{p-claim-2-a} can be written as

\begin{eqnarray}\label{p-beta-2-b} \aligned & B(\alpha,\beta;m)=\sum_{n=0}^\infty a_n, \,\, \\&\hskip 20mm{\rm where}\,\, \displaystyle a_n=\frac{(m)^n}{n!}B(\alpha+n,\beta+n).  \endaligned \end{eqnarray}
 Further

\begin{eqnarray}\label{p-beta-2-c} \aligned & \lim_{n\rightarrow \infty}\frac{a_n}{a_{n+1}}=\infty>1. \endaligned \end{eqnarray}

By ratio test,  $B(\alpha,\beta;m)$ is convergent for $m>0.$
\newline

\begin{case}\label{case-2} If $m<0$, then to prove that the extension of classical beta function $B(\alpha,\beta;m)$ is convergent. \end{case}

To prove this case, let $m=-p$ (where $p>0$), then equation \eqref{p-claim-2-a}
\begin{eqnarray}\label{p-beta-2-d} \aligned & B(\alpha,\beta;m)=\sum_{n=0}^\infty \frac{(-p)^n}{n!}B(\alpha+n,\beta+n), \endaligned \end{eqnarray}

the above equation \eqref{p-beta-2-d} can be written as
\begin{eqnarray}\label{p-beta-2-e} \aligned & B(\alpha,\beta;m)=-\sum_{n=0}^\infty (-1)^{n-1}b_n,\,\,\\&\hskip 20mm{\rm where}\,\, \displaystyle b_n=\frac{(p)^n}{n!}B(\alpha+n,\beta+n). \endaligned \end{eqnarray}

The series \eqref{p-beta-2-e} is an alternating series, therefor
\begin{enumerate}
	\item  ~$a_n>0$, $\forall p>0,\Re(\alpha,\beta)>0$
	\item ~$\displaystyle a_n-a_{n+1}=\frac{p^n}{n!}B(\alpha+n,\beta+n)~~~~~~~~~~$ $\displaystyle\times\left[1-\frac{(\alpha+n)(\beta+n)}{(\alpha+\beta+2n)(\alpha+\beta+2n+1)}\right] >0$ $\Rightarrow a_n$ is decreasing
	\item ~$\displaystyle\lim_{n\rightarrow \infty}a_n=0$ if $p\leq2$ ($\displaystyle {p^n}/{n!}\rightarrow 0$ as $n\rightarrow\infty$ only if $p\leq 2$ and  $B(\alpha+n,\beta+n)\rightarrow 0$ as $n\rightarrow \infty$)
\end{enumerate}
All the conditions of Leibniz's test have been satisfied, therefore $B(\alpha,\beta;m)$ is convergent for $0<p\leq 2\,\,\, i.e.\,-2\leq m<0$.
\newline

From the above two cases \ref{case-1} and \ref{case-2} implies that the power series in equation \eqref{p-claim-2-a} is convergent.
\end{proof}

\section{Comparison of numerical values of extended beta function, modified extension of the classical beta function and classical beta function}\label{section-3.1}
In Table 1, the columns $\rm B1, B2$ are the values of the  classical beta function $B(x,y)$; the columns $\rm EB1,$ $\rm EB2,$ $\rm EB3,$ $\rm EB4$ are the values of the extension of classical beta  

\begin{table}[ht]\caption{Comparison of extensions of classical beta function}
  \label{table-1}
  	\centerline{\includegraphics[width=5.9in, height=2.4in, keepaspectratio=false]{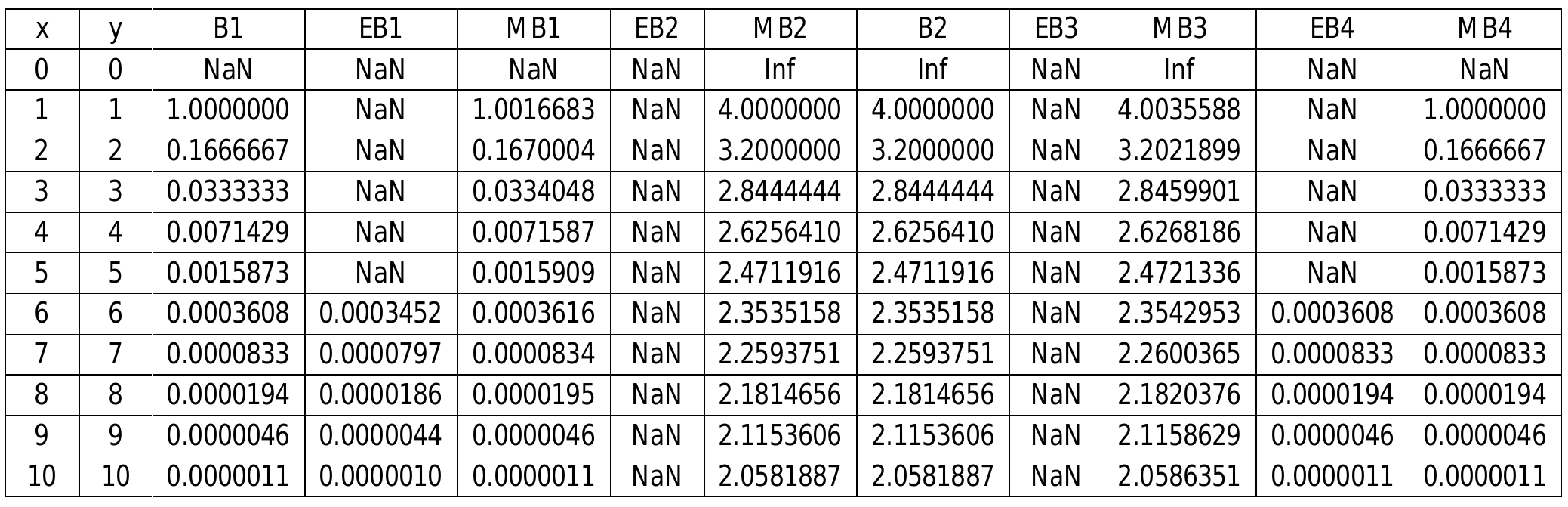}}
\end{table}
 \begin{table}[ht]\caption{Numerical Values of MECBF}
  \label{table-2}
  	\centerline{\includegraphics[width=5.9in, height=4.4in, keepaspectratio=false]{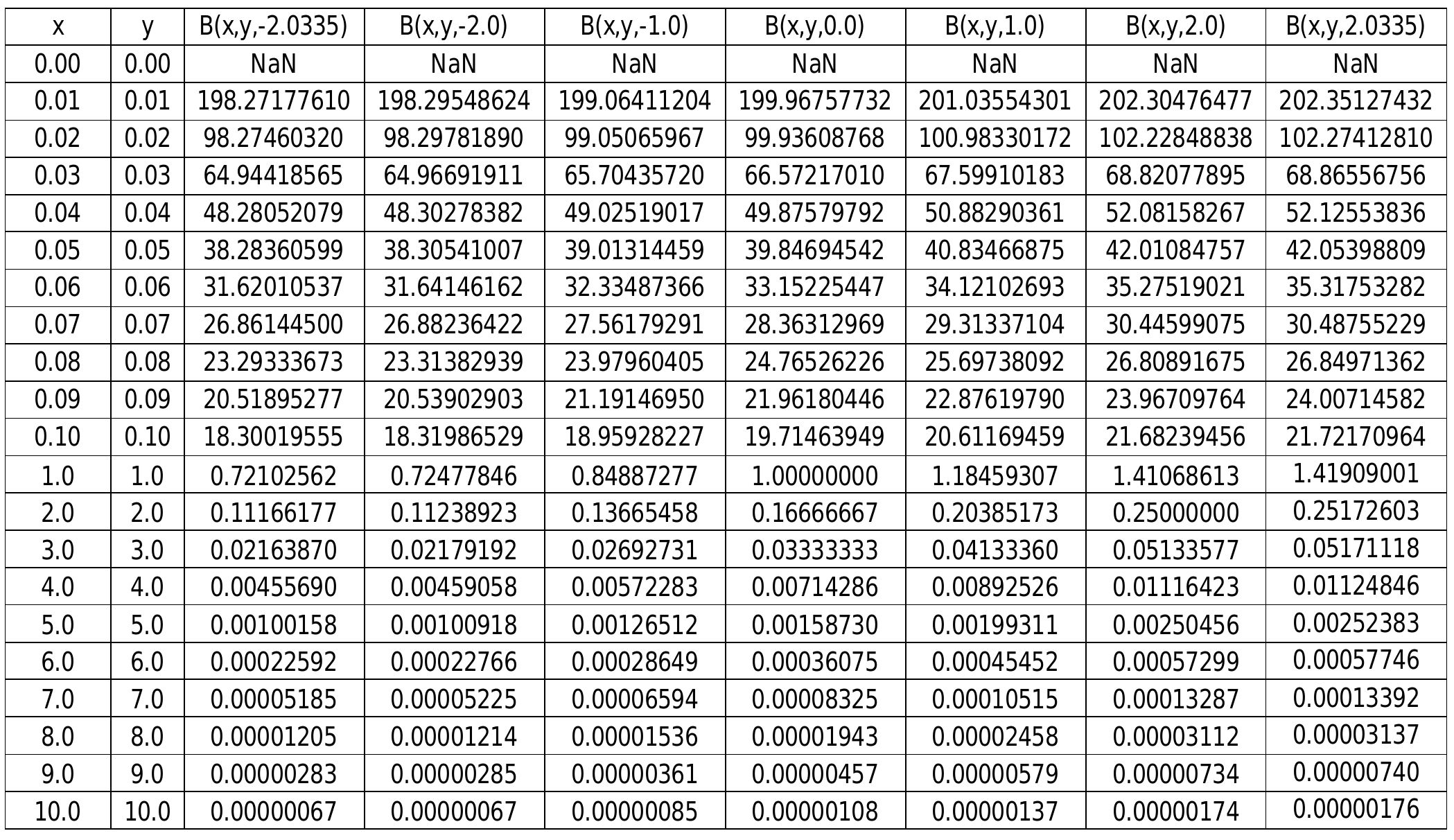}}
\end{table}

\noindent      function; the columns $\rm MB1,$ $\rm MB2,$  $\rm MB3,$ $\rm MB4$ are values of  modified extension of the  classical beta function, where $\rm  B1=$ $\rm B(x,y)$, $\rm  EB1=$ $\rm B(x,y;0.01)$, $\rm  MB1= B(x,y;0.01)$, $\rm  B2=$ $\rm B(x,0.25)$,  $\rm  EB2$  $\rm =B(x,0.25;0.00)$, $\rm  MB2=$ $\rm B(x,0.25;0.00)$, $\rm  EB3=$ $\rm B(x,0.25;0.01)$, $\rm  MB3=$ $\rm B(x,0.25;0.01)$, $\rm  EB4=$ $\rm B(x,y;0.00)$;  
    $\rm  MB4=$ $\rm B(x,y;0.00)$ and $\rm  x,y=0:1:10$.
\newline

These values are computed by employing Matlab (R2012a). The values of the extension of the classical beta function are computed by using the equation \eqref{p-claim-1-a} and those of modified extension of the classical beta function from equation \eqref{p-claim-2-a}. Both the extension in equations \eqref{p-claim-1-a} and \eqref{p-claim-2-a} are in series form involving beta functions $B(x-n,y-n)$ (exist for $\Re(x-n>0,y-n)>0$) and $B(x+n,y+n)$ (exist for $\Re(x+n>0,y+n)>0$) respectively, and the values can be computed easily. Therefor the values of extension of the classical beta function \eqref{p-claim-1-a} are  computed by taking sum of the first five terms of the series in the view of existing values of $B(x-n,y-n)$ to find some numerical values and in the case of modified extension of classical beta function, sum of first 1000 terms have been taken for $x,y=0:1:10$. The numerical results of the same  are established in Table 1.
\newline

From Table 1, it is easily observed that the values of extended beta function in columns $\rm EB1,EB4$ does not exist for $x,y=0:1:5$ (since  $\Re(x-n<0,y-n)<0$); rest values of the same exist for $x,y=6:1:10$ (since  $\Re(x-n)>0, \Re(y-n)<0$). In columns $\rm EB2,EB3$ the values of extended beta function do not exist since in these cases we choose $y=0.25$ therefor $\Re(y-n)$ becomes less than zero, which implies if $0<x,y<1$ the values of the extended beta function do not exist even if sum of first two terms of the series \eqref{p-claim-1-a} to be taken. If we take large sum of the terms of the series \eqref{p-claim-1-a} (say $n=1000$) the extended beta function will not converge for any value of $x,y\in[0,1000]$ from which we conclude that the  extended beta function does not exits for any value of $x$ and $y$, since it is series having infinite terms and $\Re(x),\Re(y)$ can not be infinite.
\newline

In the case of modified extension of classical beta function, the columns $\rm B1,MB4$ are identical, which implies that if we choose $m=0$, we have the values of classical beta function from the modified extension of the classical beta function. Column $\rm B2, MB2$ are also identical since $m=0$ and at the fix value of $y=0.25$. If we compare the values of $\rm B1$ with $\rm MB1$ and those of $\rm B2$ with $\rm MB3$, we can see the impact of the value of $m$ on the classical beta function.

\begin{note}
It is also noted that if $p=0$, the extension of classical beta function should be reduced to the classical beta function $i.e.$ $B(\alpha,\beta;0)$ $ =B(\alpha,\beta)$. But the same not depict in  the graph cited in \cite[Fig. 1, p. 31]{Chaudhry-Qadir-Rafique-Zubair-1997}. The graph  depicts the values of extension of the classical beta function $B(1,2.25;0)=$ $ B(1,0.25)=13.75$ when $p=0,x=1,$ $ y=0.25$  and $B(10,2.25;0)=$ $ B(10,0.25)=2.18$(approximately) when $p=0,$ $ x=10,$ $ y=0.25$, which are not the correct values of the classical beta function. The correct values of the classical beta function at these point are $B(1,0.25)=4$ and $B(10,0.25)$ $ =2.0582$. It also can be easily observed from the figure that when $p\neq 0$ the graph of extended beta function totally different from that of the classical beta function. $i.e.$ the behavior and nature of the curves differ from that of classical beta function.
\end{note}
\begin{note} From the above discussion, it is very clear that  the extension of classical beta function in equation \eqref{ex-beta} do not convergent.  The extensions in equations \eqref{B1}, \eqref{B2} and \eqref{B3} are further extension of the extended beta function \eqref{ex-beta}, therefor all these extensions also can not be convergent.
\end{note}

\begin{note}
  Due to the lack of the convergence of the extension of the classical beta function, all the associated results established by Chaudhary $et\,\, al.$ \cite{Chaudhry-Qadir-Rafique-Zubair-1997} do not convergent.
\end{note}

\section{Numerical results and their interpretation of modified extension of the classical beta function}\label{section-3.2}
The numerical results of modified extension of classical beta function (MECBF) have been calculated in this section by employing the Matlab. We choose the values of variable $x,y$ and parameter $m$ as $x,y\in [0,10]$ and $m\in[-2.0335,2.0335]$. All the numerical values of MECBF are presented in Table 1, from which we can easily observe that $B(x,y;m)$ does not exist at $x=y=0$ and it is also tested that $B(x,y;m)$ does not exist for $m<-2.0335$ and $m>2.0335$. It can be easily investigated that $B(x,y;m)\rightarrow \infty$ as $x,y\rightarrow 0$ and $B(x,y;m)\rightarrow 0$ as $x,y\rightarrow \infty$, which implies that the behavior of modified extension of classical beta function is the same as that of classical beta function, which can be observed from Table 2 and Figure 1.   Also $B(x,y;0)$ represents the values of classical beta function $B(x,y)$.

 \begin{figure}[ht]
  \label{figure-1}
  	\centerline{\includegraphics[width=4.5in, height=4in, keepaspectratio=false]{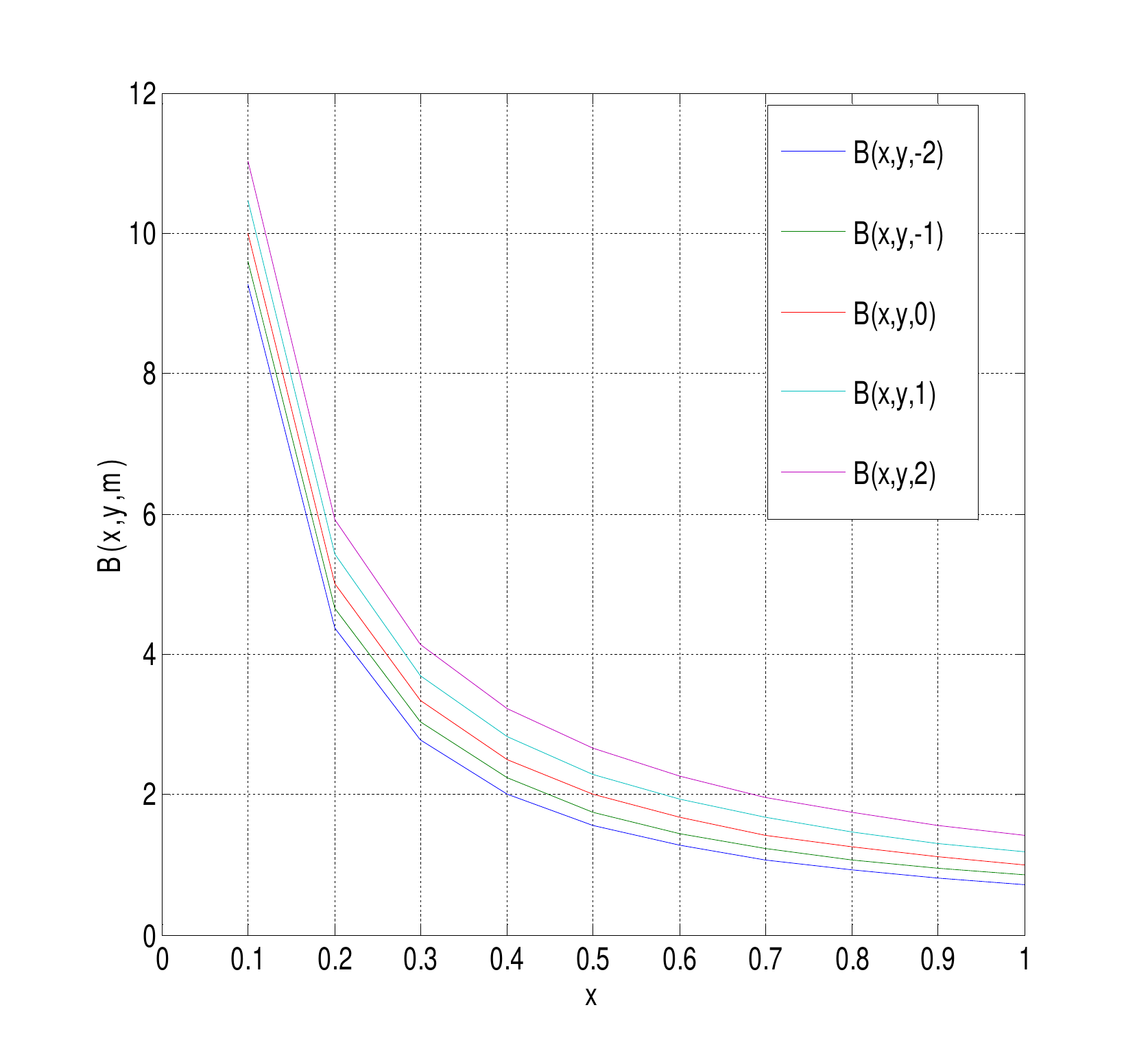}}
		\caption{Graphical presentation of of MECBF for y=1.}
\end{figure}


\section{Geometrical interpretation of modified extension of the classical beta function}\label{section-3.3}
 If  $\beta=1,$ $ m\in[-2,2]$ and $x\in[0,1]$, then we observe that graph of $B(x,y;m)$ in Figure 1 is decreasing. In figure 2, we choose $m=2$ and $x,y\in[0, 25]$, which depicts that as $x,y\rightarrow 0$, $B(x,y;m)\rightarrow\infty$ and $B(x,y;m)\rightarrow 0$ as $\alpha,\beta\rightarrow \infty$. We also check the effect of $m$ on the modified extension of classical beta function. For this purpose, we fix the values of $x$ and $y$ as shown in figure 3, then we plot the graph which depicts that $B(x,y;m)$ is an increasing function as the values of $m$ increase. It is very clear from Figure 1 that for the graph of classical beta function and modified extension of classical beta function remains concave upward (or convex downward) for different values of $\alpha,\beta$ and $m$. The value of $m$ does not effect the nature of classical beta function, the main effect of the value of $m$ is that it just push the curve up or drag down the curve from the curve of the classical beta function as shown in the Figure 1 the same behavior can be observed from the Table 2.

\begin{figure}[ht]
  \label{figure-2}
  	\centerline{\includegraphics[width=4.5in, height=4in, keepaspectratio=false]{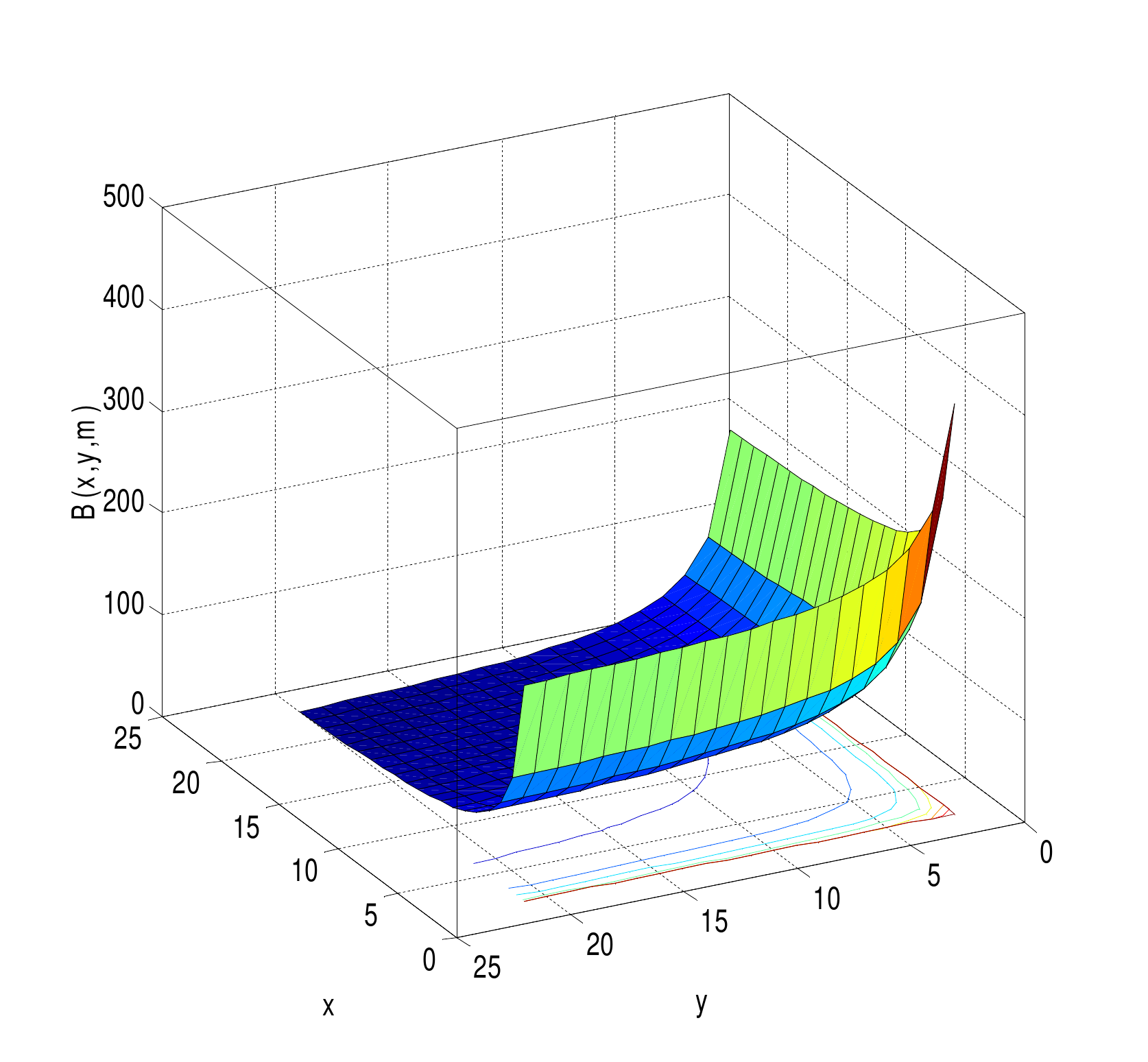}}
		\caption{3D presentation of of MECBF.}
\end{figure}


From the above proof of radius of convergence of series and further more numerically investigation of the power series in Table 1, we find that the interval of convergence of the series is $(-2.0336,2.0336)$, which implies that $B(x,y;m)$ is convergent for $|m|<M$ where $M$ is positive real number slightly greater than 2.0335.

\begin{note}From the above discussion, it is easy to conclude that the value of  $\Re(m)$ lies in the interval $[-2.0335,2.0335]$ $i.e.$ $-2.0335\leq\Re(m)\leq 2.0335$.\end{note}
\begin{note}
In the sequel of this paper, $|m|<M$ represents the circle of convergence and $M$  is the radius of convergence of equation \eqref{ex-beta-m}, where $M$ is sightly grater than from $2.0335.$
\end{note}


\begin{figure}[ht]
  \label{figure-3}
  	\centerline{\includegraphics[width=4.5in, height=4in, keepaspectratio=false]{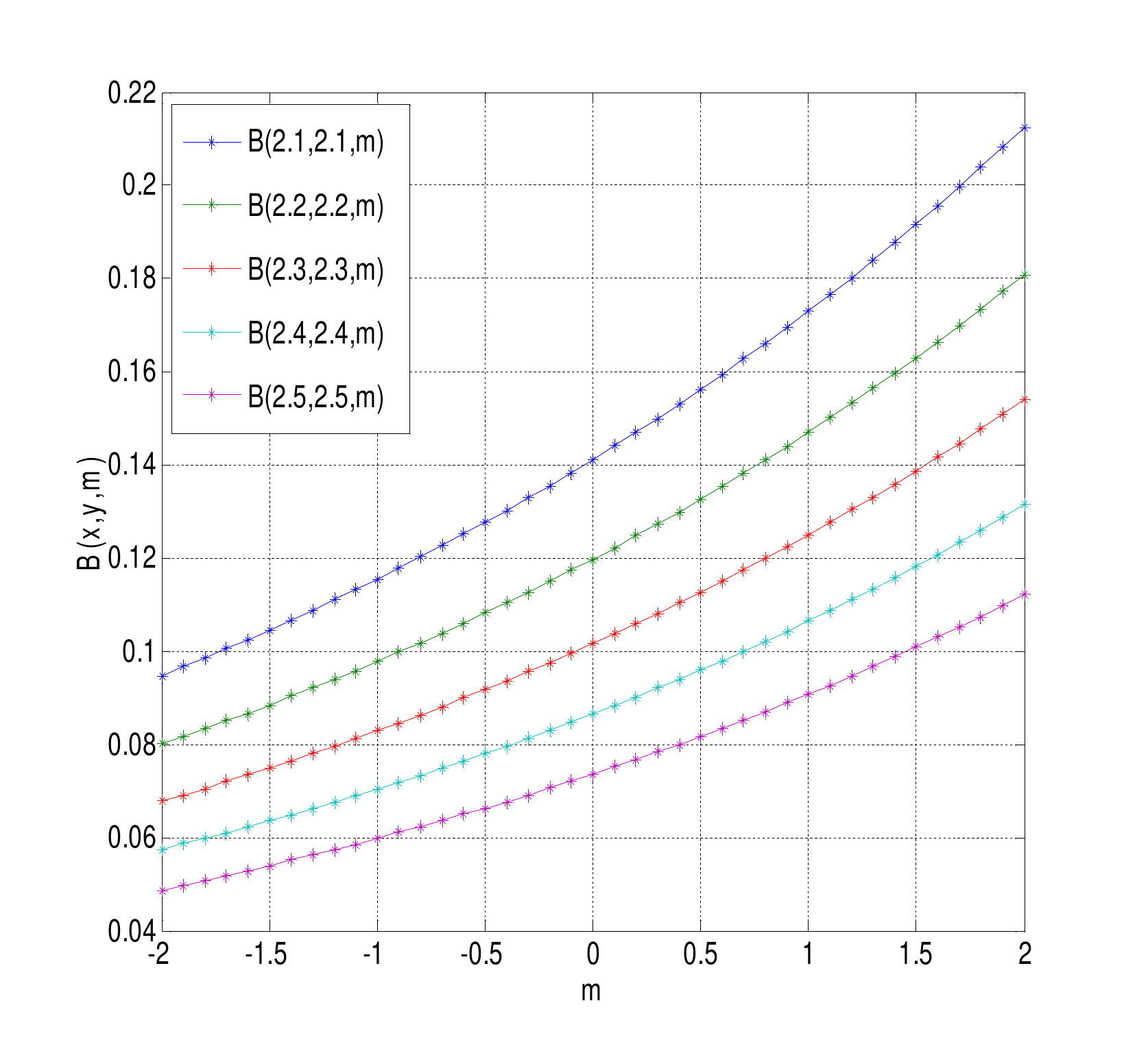}}
		\caption{Graph of MECBF for fixed value of $x,y$ and $m=-2:.5:2$. }
\end{figure}


\section{Integral representation of the MECBF function}

The integral representation of the modified extended beta function is important to check both the extension is natural and simple and for use later. It is also important to investigate the relationship between the classical beta function and  the modified extension of the classical beta function. In this connection we first provide a relationship between them. The following integral formula is useful for further investigation \cite{grad-ryzhik}
\begin{eqnarray}\label{int-for-1} \aligned & \int_0^\infty x^m\exp(-\beta{}x^n)dx=\frac{\Gamma(\gamma)}{n\beta^\gamma},\hskip 2mm \left( \gamma=\frac{m+1}{n}\right). \endaligned \end{eqnarray}

\begin{theorem}[Relation between modified extension of the classical beta function and the classical beta function] If $\Re(\alpha+s)>0,$ $\Re(\beta+s)>0,$ $ m\in\mathbb{C};$ $|m|<M$ (where $M$ is positive finite real number slightly greater than 2.0335), then we have the following relation
\begin{eqnarray}\label{relation-1} \aligned & \int_0^\infty m^{s-1}B(\alpha,\beta;m)dm=(-1)^s\Gamma(s)B(\alpha+s,\beta+s). \endaligned \end{eqnarray}
\end{theorem}

\begin{proof} Multiplying both sides of equation \eqref{ex-beta-m} by $m^{s-1}$, then integrate w.r.t. $m$ from $m=0$ to $m=\infty$, we have
\begin{eqnarray}\label{relation-1-pro-1} \aligned & \int_0^\infty m^{s-1}B(\alpha,\beta;m)dm\\&=\int_0^\infty m^{s-1}\left[\int_0^1 t^{\alpha-1}(1-t)^{\beta-1}\exp(mt(1-t))dt\right]dm, \endaligned \end{eqnarray}

interchanging the order of integration, the above equation \eqref{relation-1-pro-1}, reduces to

\begin{eqnarray}\label{relation-1-pro-2} \aligned & \int_0^\infty m^{s-1}B(\alpha,\beta;m)dm\\&=\int_0^1 t^{\alpha-1}(1-t)^{\beta-1}\left[\int_0^\infty m^{s-1}\exp(mt(1-t))dm\right]dt, \endaligned \end{eqnarray}

further using the formula given in equation \eqref{int-for-1}, after simplification the above equation \eqref{relation-1-pro-2} reduces to

\begin{eqnarray}\label{relation-1-pro-3} \aligned & \int_0^\infty m^{s-1}B(\alpha,\beta;m)dm\\&=(-1)^s\Gamma(s)\int_0^1 t^{\alpha+s-1}(1-t)^{\beta+s-1}dt, \endaligned \end{eqnarray}
using the definition of classical beta function, we have the required result.
\end{proof}

\begin{remark}
By setting $s=1$, the result in equation \eqref{relation-1} reduces to
\begin{eqnarray} \aligned & \int_0^\infty B(\alpha,\beta;m)dm=-B(\alpha+1,\beta+1), \\&\hskip 20mm \Re(\alpha)>-1,\Re(\beta)>-1,\endaligned \end{eqnarray}
gives the interesting relation between classical beta function and modified extended beta function.
\end{remark}

\begin{remark}
All the derivatives of the modified extension of classical beta function (MECBF) with respect to the parameter $m$ can be expressed in terms of the function  as
\begin{eqnarray} \aligned & \frac{\partial^n }{\partial^nm}B(\alpha,\beta;m) =B(\alpha+n,\beta+n),\\& \hskip 20mm \Re(\alpha+n)>0,\Re(\beta+n)>0.\endaligned \end{eqnarray}

\end{remark}

\begin{theorem}[Integral representations of the modified extension of the classical beta function] If $\Re(x)>0,$ $\Re(y)>0, $ $m\in\mathbb{C};$ $|m|<M$ (where $M$ is positive finite real number slightly greater than 2.0335), then we have the following relation

\begin{eqnarray}\label{Bxy-1} \aligned & B(\alpha,\beta;m)=2\int_0^{\frac{\pi}{2}}\cos^{2\alpha-1}\theta\sin^{2\beta-1}\theta\\&\hskip20mm\times\exp(m\cos^2\theta\sin^2\theta)d\theta, \endaligned \end{eqnarray}

\begin{eqnarray}\label{Bxy-2} \aligned & B(\alpha,\beta;m)=\int_0^{\infty}\frac{u^{\alpha-1}}{(1+u)^{\alpha+\beta}}\exp\left({mu}/{(1+u)^2}\right)du,\endaligned \end{eqnarray}

\begin{eqnarray}\label{Bxy-3} \aligned & B(\alpha,\beta;m)=\frac{1}{2}\int_{0}^{\infty}\frac{u^{\alpha-1}+u^{\beta-1}}{(1+u)^{\alpha+\beta}}\\&\hskip20mm\times\exp\left({mu}/{(1+u)^2}\right)du,\endaligned \end{eqnarray}

\begin{eqnarray}\label{Bxy-4} \aligned & B(\alpha,\beta;m)=(c-a)^{1-\alpha-\beta}\int_{a}^{c}(u-a)^{\alpha-1}(c-u)^{\beta-1}\\&\hskip20mm\times\exp\left(\frac{m(u-a)(c-u)}{(c-a)^2}\right)du,\endaligned \end{eqnarray}

\begin{eqnarray}\label{Bxy-5} \aligned & B(\alpha,\beta;m)=2^{1-\alpha-\beta}\int_{-1}^{1}(1+t)^{\alpha-1}(1-t)^{\beta-1}\\&\hskip20mm\times\exp\left(m(1-t^2)/4\right)du,\endaligned \end{eqnarray}

\begin{eqnarray}\label{Bxy-6} \aligned & B(\alpha,\beta;m)=2^{1-\alpha-\beta}\\&\times\int_{-\infty}^{\infty}\exp\left[(\alpha-\beta)x+\frac{m}{4\cosh^2x}\right]\frac{dx}{(\cosh x)^{\alpha+\beta}},\endaligned \end{eqnarray}

\begin{eqnarray}\label{Bxy-7} \aligned & B(\alpha,\beta;m)=2^{2-\alpha-\beta}\int_{0}^{\infty}\cosh((\alpha-\beta)x)\\&\hskip20mm\times\exp\left[\frac{m}{4\cosh^2x}\right]\frac{dx}{(\cosh x)^{\alpha+\beta}},\endaligned \end{eqnarray}

\begin{eqnarray}\label{Bxy-8} \aligned & B(\alpha,\beta;m)=2^{1-\alpha-\beta}\\&\times\int_{-\infty}^{\infty}\exp\left[\frac{1}{2}(\alpha-\beta)x+\frac{m}{2\cosh x}\right]\frac{dx}{(\cosh x/2)^{\alpha+\beta}},\endaligned \end{eqnarray}

\begin{eqnarray}\label{Bxy-9} \aligned & B(\alpha,\beta;m)=2^{2-\alpha-\beta}\int_{0}^{\infty}\cosh((\alpha-\beta)x/2)\\&\hskip20mm\times\exp\left[\frac{m}{2\cosh x}\right]\frac{dx}{(\cosh x/2)^{\alpha+\beta}}.\endaligned \end{eqnarray}

 \end{theorem}

\begin{proof}
The result \eqref{Bxy-1} can be easily obtained by setting $t=\cos^2\theta$, to prove \eqref{Bxy-2} choose $t=u/(1+u)$, \eqref{Bxy-3} can be easily obtained by applying the symmetric property in equation \eqref{Bxy-2} then adding new one and \eqref{Bxy-2}, the result in equation \eqref{Bxy-4} is obtained by taking $t=(u-a)/(c-a)$, setting $a=-1,c=1$ in equation \eqref{Bxy-4} gives the result in equation \eqref{Bxy-5} and to prove the result in equation \eqref{Bxy-6} put $u=\tanh x$ in \eqref{Bxy-5}. The results in equation \eqref{Bxy-7}, \eqref{Bxy-8} and \eqref{Bxy-9} can be easily obtained from the result \eqref{Bxy-6}.
\end{proof}

\begin{remark}[Useful inequalities] If $\Re(\alpha)>0,$ $\Re(\beta)>0$, then we have the following inequality
\begin{eqnarray}\label{equlity-1} \aligned & |B(\alpha,\beta;m)|\leq 1.6626B(\alpha,\beta)\endaligned \end{eqnarray}
follows from the integral representation \eqref{Bxy-2}. Since the function $\exp(mu/(1+u)^2)$ attains its maximum value $1.6626$ at $u=1$ and $m=2.0335$. The equality is verified with the help of numerical results by using Matlab.
 \end{remark}


 \section{Properties of the modified extension of the classical beta function}

\begin{theorem}[Functional Relation] If $\Re(\alpha)>0,$ $\Re(\beta)>0, $ $m\in\mathbb{C};$ $|m|<M$ (where $M$ is positive finite real number slightly greater than 2.0335), then we have the following relation

\begin{eqnarray}\label{functional-relation} \aligned & B(\alpha,\beta+1;m)+B(\alpha+1,\beta;m)= B(\alpha,\beta;m). \endaligned \end{eqnarray}
 \end{theorem}

\begin{proof} LHS of the above equation \eqref{functional-relation} is equal to

\begin{eqnarray}\label{fr-proof}\int_0^1\{ t^{\alpha-1}(1-t)^{\beta}+t^{\alpha}(1-t)^{\beta-1}\}e^{mt(1-t)}dt \end{eqnarray}
after simplification the above equation \eqref{fr-proof} reduced to

\begin{eqnarray}\label{fr-proof-1}\int_0^1\{ t^{\alpha-1}(1-t)^{\beta-1}\}e^{mt(1-t)}dt=B(\alpha,\beta;m). \end{eqnarray}
\end{proof}
If we choose $m=0$, we get the usual relation for the beta function from \eqref{functional-relation}.

\begin{theorem}[Symmetry] If $\Re(\alpha)>0,$ $\Re(\beta)>0, $ $m\in\mathbb{C};$ $|m|<M$ (where $M$ is positive finite real number slightly greater than 2.0335), then we have the following relation

\begin{eqnarray}\label{symmetry} \aligned & B(\alpha,\beta;m)= B(\beta,\alpha;m). \endaligned \end{eqnarray}
 \end{theorem}

\begin{proof} From equation \eqref{p-claim-2-a} we have

\begin{eqnarray}\label{sy-proof} \aligned & B(\alpha,\beta;m)=\sum_{n=0}^\infty \frac{(m)^n}{n!}B(\alpha+n,\beta+n)  \endaligned \end{eqnarray}
 Since usual beta function is symmetric $i.e.$ $B(\alpha,\beta)=B(\beta,\alpha)$. Using this property in the right hand side of the equation \eqref{sy-proof}, then we have

\begin{eqnarray}\label{sy-proof-1} \aligned & B(\alpha,\beta;m)=\sum_{n=0}^\infty \frac{(m)^n}{n!}B(\beta+n,\alpha+n)= B(\beta,\alpha;m). \endaligned \end{eqnarray}
\end{proof}

\begin{theorem}[First Summation Relation] If $\Re(\alpha)>0,$ $\Re(1-\beta)>0,$ $ m\in\mathbb{C};$ $|m|<M$ (where $M$ is positive finite real number slightly greater than 2.0335), then we have the following relation

\begin{eqnarray}\label{summation-1} \aligned & B(\alpha,1-\beta;m)= \sum_{n=0}^\infty \frac{(\beta)_n}{n!}B(\alpha+n,1;m). \endaligned \end{eqnarray}
 \end{theorem}
\begin{proof} The LHS of the equation \eqref{summation-1} can be written as
\begin{eqnarray}\label{summation-p-1} \aligned & B(\alpha,1-\beta;m)= \int_0^1 t^{\alpha-1}(1-t)^{-\beta}\exp(mt(1-t))dt, \endaligned \end{eqnarray}
using the binomial series expansion $\displaystyle(1-t)^{-\beta}=$ $\displaystyle\sum_{n=0}^\infty\frac{(\beta)_n}{n!}t^n$ in the above equation \eqref{summation-p-1} and then interchanging the order of summation and integration, the above result \eqref{summation-p-1} reduced to the following form

\begin{eqnarray}\label{summation-p-2} \aligned & B(\alpha,1-\beta;m)\\&= \sum_{n=0}^\infty\frac{(\beta)_n}{n!}\int_0^1 t^{\alpha+n-1}\exp(mt(1-t))dt \endaligned \end{eqnarray}

\begin{eqnarray}\label{summation-p-3} \aligned & \Rightarrow B(\alpha,1-\beta;m)= \sum_{n=0}^\infty\frac{(\beta)_n}{n!}B(\alpha+n,1;m). \endaligned \end{eqnarray}

\end{proof}

\begin{theorem}[Second Summation Relation] If $\Re(\alpha)>0,\Re(\beta)>0, m\in\mathbb{C};|m|<M$ (where $M$ is positive finite real number slightly greater than 2.0335), then we have the following relation

\begin{eqnarray}\label{summation-2} \aligned & B(\alpha,\beta;m)= \sum_{n=0}^\infty B(\alpha+n,\beta+1;m). \endaligned \end{eqnarray}
 \end{theorem}
\begin{proof} The LHS of the equation \eqref{summation-1} can be written as
\begin{eqnarray}\label{summation-p-21} \aligned & B(\alpha,\beta;m)= \int_0^1 t^{\alpha-1}(1-t)^{\beta-1}\exp(mt(1-t))dt, \endaligned \end{eqnarray}
using the binomial series expansion $\displaystyle(1-t)^{\beta-1}=(1-t)^\alpha\sum_{n=0}^\infty t^n,\hskip 2mm(|t|<1)$ and interchanging the order of summation and integration, the above equation \eqref{summation-p-21} reduces to

\begin{eqnarray}\label{summation-p-22} \aligned & B(\alpha,\beta;m)= \sum_{n=0}^\infty\int_0^1 t^{\alpha+n-1}(1-t)^\beta\exp(mt(1-t))dt \endaligned \end{eqnarray}

\begin{eqnarray}\label{summation-p-23} \aligned & \Rightarrow B(\alpha,\beta;m)= \sum_{n=0}^\infty B(\alpha+n,\beta+1;m). \endaligned \end{eqnarray}
\end{proof}

\begin{theorem}[Separation] If $\Re(\alpha)>0,$ $\Re(\beta)>0, $ $m\in\mathbb{C};$ $|m|<M$ (where $M$ is positive finite real number slightly greater than 2.0335), then $B(\alpha,\beta;m)$ can can be separated into real and imaginary part of $m$ as follows
\begin{eqnarray}\label{real} \aligned & B(\alpha,\beta;r\cos\theta)\\&=\int_0^1t^{\alpha-1}(1-t)^{\beta-1}\exp(rt(1-t)\cos\theta)dt,\endaligned \end{eqnarray}
\begin{eqnarray}\label{imag} \aligned & B(\alpha,\beta;r\sin\theta)\\&=\int_0^1t^{\alpha-1}(1-t)^{\beta-1}\exp(rt(1-t)\sin\theta)dt, \endaligned \end{eqnarray}
where $r=\sqrt{x^2+y^2}=|m|<M$ and $\theta=\tan^{-1}y/x.$
 \end{theorem}
 \begin{proof}
 Since $m\in\mathbb{C}$, so let $m=x+iy$ where $x,y\in\mathbb{R}$ and also let $x+iy=r\cos\theta+ir\sin\theta$ $\Rightarrow$ $r=\sqrt{x^2+y^2}$ and $\theta=\tan^{-1}y/x$, then from equation \eqref{ex-beta-m}, we have
 \begin{eqnarray}\label{real-1} \aligned & B(\alpha,\beta;x+iy)\\&=\int_0^1t^{\alpha-1}(1-t)^{\beta-1}\exp(re^{i\theta}t(1-t))dt, \endaligned \end{eqnarray}
 after simplification the above equation \eqref{real-1} reduces to

 \begin{eqnarray}\label{real-2} \aligned & B(\alpha,\beta;r\cos\theta+ir\sin\theta)\\&=\int_0^1t^{\alpha-1}(1-t)^{\beta-1}\exp(rt(1-t)\cos\theta)dt\\&
 \hskip 10mm +i\int_0^1t^{\alpha-1}(1-t)^{\beta-1}\exp(rt(1-t)\sin\theta)dt, \endaligned \end{eqnarray}
 equating the real and imaginary parts of $m$ only, we have the required results.
 \end{proof}

\section{The modified extended beta distribution}
It is expected that there will be many applications of the modified extension of the classical  beta function, like there were of the generalized gamma function. One application that springs to mind is to Statistics. For example, the conventional beta distribution can be extended, by using our modified extension of the classical  beta function, to variables p and q with an infinite range. It appears that such an extension may be desirable for the project evaluation and review technique used in some special cases.
\newline

 We define the extended beta distribution by

\begin{eqnarray}\label{mod-ex-beta-distr-1} \aligned &  f(t)=\left\{
\begin{array}{cc}
	\displaystyle\frac{t^{p-1}(1-t)^{q-1}\exp(mt(1-t))}{B(p,q;m)},\hskip 2mm 0<t<1
	\\{}\\
	0,\hskip 20mm {\rm otherwise}.
\end{array}
 \right.  \endaligned \end{eqnarray}

A random variable $X$ with probability density function (pdf) given in equation \eqref{mod-ex-beta-distr-1} will be said to have the extended beta distribution with parameters $p$ and $q$, $-\infty<p,q<\infty$ and $|m|<M$ where $M$ is positive real number slightly greater than 2.0335. If $\nu$ is any real number \cite{rohatgi-1976}, then

\begin{eqnarray}\label{mod-ex-beta-distr-2} E(X^\nu)=\frac{B(p+\nu,q;m)}{B(p,q;m)}. \end{eqnarray}

In particular, for $\nu=1$,

\begin{eqnarray}\label{mod-ex-beta-distr-3}\mu= E(X)=\frac{B(p+1,q;m)}{B(p,q;m)} \end{eqnarray}

represents the mean of the distribution and

\begin{eqnarray}\label{mod-ex-beta-distr-4}\aligned & \sigma^2 =E(X^2)-(E(X))^2\\&=\frac{B(p,q;m)B(p+2,q;m)-B^2(p+1,q;m)}{B^2(p,q;m)} \endaligned \end{eqnarray}

is a variance of the  distribution.
\newline

The moment of generating function of the distribution is

\begin{eqnarray}\label{mod-ex-beta-distr-5}\aligned & M(t)=\sum_{n=0}^\infty \frac{t^nE(X^n)}{n!}\\&=\frac{1}{B(p,q;m)}\sum_{n=0}^\infty\frac{t^n}{n!}B(p+n,q;m). \endaligned \end{eqnarray}
The commutative distribution of \eqref{mod-ex-beta-distr-1} can be written as

\begin{eqnarray}\label{mod-ex-beta-distr-6} F(x)=\frac{B_{x}(p,q;m)}{B(p,q;m)} \end{eqnarray}

where

\begin{eqnarray}\label{mod-ex-beta-distr-7}\aligned & B_{x}(p,q;m)=\int_0^xt^{p-1}(1-t)^{q-1}\exp(mt(1-t))dt,\\ & \hskip 40mm|m|<M, -\infty<p,q<\infty \endaligned \end{eqnarray}

is the modified extended incomplete beta function. For $m=0$, we must have $p,q>0$ in \eqref{mod-ex-beta-distr-7} for convergence, and $B_{x}(p,q;0)=B_x(p,q)$, where $B_x(p,q)$ is the incomplete beta function \cite{grad-ryzhik} defined as

\begin{eqnarray}\label{mod-ex-beta-distr-8}\aligned & B_{x}(p,q)=B_x(p,q;0)=\frac{x^p}{p}{}_2F_1(p,1-q;p=1;x). \endaligned \end{eqnarray}

It is to be noted that the problem of expressing $B_x(p,q,m)$ in terms of other special functions remains open. Presumably, this distribution should be useful in extending the statistical results for strictly positive variables to deal with variables that can take arbitrarily large negative values as well.

\section{conclusion}

The (Euler's) classical beta function $B(x; y)$ play an important role in various fields in the
mathematical, physical, engineering and statistical sciences. The extensions of the classical beta function is also important for further investigation in the respective areas, provided these extensions of beta function should be convergent for each parameter involving. The modified extension of classical beta function is convergent and if we choose $m=0$, the classical beta function is obtained. It is very clear from Figure 1 that for the graph of classical beta function and modified extension of classical beta function remains concave upward (or convex downward) for different values of $\alpha,\beta$ and $m$. The value of $m$ does not effect the nature of classical beta function, the main effect of the value of $m$ is that it just push the curve up or drag down the curve from the curve of the classical beta function as shown in the Figure 1 the same behavior can be observed from the Table 2. Both modified extension of classical beta function and classical beta function $i.e.$ $B(\alpha,\beta;m)$ $\&$ $ B(\alpha,\beta)\rightarrow 0$ as $x,y\rightarrow \infty$; $B(\alpha,\beta;m)$ $\&$ $ B(\alpha,\beta)\rightarrow \infty$ as $x,y\rightarrow 0$ and both $B(\alpha,\beta;m)$ $\&$ $ B(\alpha,\beta)$ does not exist at $x=y=0$.


\emergencystretch=\hsize

\begin{center}
\rule{6 cm}{0.02 cm}
\end{center}


\begin{thebibliography}{99}

\bibitem{Srivastava-etinkaya-K?ymaz-2014} H. M. Srivastava, A. Cetinkaya, and $\dot{\rm I}$. O. Kiymaz, A certain generalized Pochhammer symbol and its applications to hypergeometric functions, Appl. Math. Comput. 226 (2014), 484-491.

\bibitem{Parmar-2013} R. K. Parmar, A new generalization of Gamma, Beta, hypergeometric and confluent hypergeometric functions, Matematiche (Catania) 69 (2013), no. 2, 33-52.

\bibitem{Ozergin-Ozerslan-Altin-2011} E. $\ddot{O}$zergin, M. A. $\ddot{O}$zarslan, and A. Altin, Extension of gamma, beta and hypergeometric functions, J. Comput. Appl. Math. 235 (2011), no. 16, 4601-4610.

\bibitem{Chaudhry-Qadir-Rafique-Zubair-1997} M. A. Chaudhry, A. Qadir, M. Rafique and S. M. Zubair, Extension of Euler's beta function, J. Comput. Appl. Math. 78 (1997), no. 1, 19-32.

\bibitem{Rainville-1971} E. D. Rainville, Special Functions, Macmillan Company, New York, 1960; Reprinted by Chelsea Publishing Company, Bronx, New York, 1971.

\bibitem{Srivastava-Choi-2012} H. M. Srivastava and J. Choi, Zeta and q-Zeta Functions and Associated Series and Integrals, Elsevier Science Publishers, Amsterdam, London and New York, 2012.

\bibitem{grad-ryzhik} I.S. Gradshteyn and I.M. Ryzhik, Table of Integrals, Series, and Products, Seventh edition, Elsevier Acadmic Press publication.

\bibitem{rohatgi-1976} V.K. Rohatgi. An Introduction to Probability Theory and Mathematical Statistics (Wiley, New York, 1976).

\bibitem{green-1958} J.A. Green. Sequence and Series, Routledge $\&$ Kegan Paul Ltd., Broudway House, Carter Lane, London (1958).
\end{thebibliography}
\end{document}